\documentclass[12pt,a4paper, reqno]{amsart}
\usepackage{geometry}
\usepackage[latin1]{inputenc}
\usepackage[english]{babel}
\usepackage{amsmath}
\usepackage{amsfonts}
\usepackage{amssymb}
\usepackage{stmaryrd}
\usepackage{amsthm}
\usepackage{mathrsfs}
\usepackage{mathabx}
\usepackage{enumitem}
\usepackage{mathtools}
\usepackage[colorlinks=true,linkcolor=red]{hyperref}
\usepackage[all]{hypcap}
\newcommand{\N}{\mathbb{N}}

\DeclarePairedDelimiter\ceil{\lceil}{\rceil}
\DeclarePairedDelimiter\floor{\lfloor}{\rfloor}

\newtheorem{theorem}{Theorem}
\newtheorem{lemma}[theorem]{Lemma}

\newtheorem{corollary}[theorem]{Corollary}

\mathcode`l="8000
\begingroup
\makeatletter
\lccode`\~=`\l
\DeclareMathSymbol{\lsb@l}{\mathalpha}{letters}{`l}
\lowercase{\gdef~{\ifnum\the\mathgroup=\m@ne \ell \else \lsb@l
\fi}}\endgroup

\begin{document}
\title{Ramsey properties of randomly perturbed dense graphs}
\thanks{A large part of this work forms the Bachelor thesis of the author which was supervised by Mathias Schacht (schacht@math.uni-hamburg.de) and handed in at the Department of Mathematics at the University of Hamburg in March 2018. While we prepared this manuscript, Das and Treglown \cite{DasTreglown} independently obtained similar and more general results (see Section \ref{ourresults})}.
\author{Emil Powierski}
\email{emil.powierski@studium.uni-hamburg.de}

\begin{abstract}
We investigate Ramsey properties of a random graph model in which random edges are added to a given dense graph. Specifically, we determine lower and upper bounds on the function $p=p(n)$ that ensures that for any dense graph~$G_n$ a.a.s.\ every 2-colouring of the edges of $G_n\cup G(n,p)$ admits a monochromatic copy of the complete graph $K_r$. These bounds are asymptotically sharp for the cases when $r\geq 5$ is odd and almost sharp when $r \geq 4$ is even. 
Our proofs utilise recent results on the threshold for asymmetric Ramsey properties in $G(n,p)$ and the method of dependent random choice.
\end{abstract}

\maketitle 

\section{Introduction} 

\subsection{Random graphs and randomly perturbed dense graphs}
For $n \in \N$ and $0\leq p \leq 1$ we denote by $G(n,p)$ the binomial random graph on~$n$ vertices where every edge is present with probability $p$ independently of all other choices. As usual, we say that an event happens \textit{asymptotically almost surely (a.a.s.)} if it holds with probability tending to $1$ as $n \rightarrow \infty$.
Given a graph property $\mathcal{P}$, it has been a key question to find a \textit{threshold function}, a function $p^{*} \colon \N \to [0,1]$ ensuring that $G(n,p)$ a.a.s.\ satisfies $\mathcal{P}$ when $p=\omega(p^*)$ and a.a.s.\ does not satisfy $\mathcal{P}$ when $p=o(p^*).$ 

Bohman, Frieze and Martin \cite{bohman2003many} considered a model that combines deterministic graphs and random graphs:
In that model of randomly perturbed graphs one starts with an arbitrary dense graph and adds edges in a random manner. 
More precisely, given $\gamma>0$, we say that a graph $G=(V,E)$ is \textit{$\gamma$-dense} if $|E| \geq \gamma |V|^2.$  Furthermore, we say that $(\gamma,p)$ \textit{ensures} a property $\mathcal{A}$ if
$$\tau_p^{\mathcal{A}}= \lim_{n \to \infty} \min_{G_n} \, \Pr(G_n \cup G(n,p(n)) \text{ satisfies } \mathcal{A})=1,$$
where the minimum is taken over all $\gamma$-dense graphs on the same vertex set as $G(n,p(n))$.
For a fixed $\gamma>0$, we say that a function $p^*$ is a \textit{threshold} for $\mathcal{A}$ (in the context of randomly perturbed dense graphs) if $\tau_p^{\mathcal{A}}=1$ for $p=\omega(p^*)$ and $\tau_p^{\mathcal{A}}=0$ for $p=o(p^*).$ Throughout, we will assume that $\gamma>0$ is some fixed and small constant.

Stricly speaking, working in this model requires to consider sequences of $\gamma$-dense graphs $(G_n)_{n \in \N}$. However, for a better presentation, we suppress the sequences and similarly we simply write $p$ for $p(n)$.

Recently, several thresholds for this model have been studied in \cite{bohman2004adding,krivelevich2006smoothed, bottcher2017embedding, krivelevich2017bounded, balogh2018tilings, krivelevich2016cycles, bedenknecht2018powers, bennett2017adding, bottcher2018universality, han2018hamiltonicity, joos2018spanning, mcdowell2018hamilton, dudek2018powers}.
Most of the analysis centered around ensuring spanning structures such as trees or (powers of) cycles. 
Krivelevich, Sudakov and Tetali \cite{krivelevich2006smoothed} already investigated  Ramsey properties of this model (see Section \ref{subsRP} below). We continue this line of research (see Section~\ref{ourresults}).

\subsection{Ramsey properties of random graphs}
For graphs $G, H_1, \dots ,H_k$, we denote by $G \rightarrow (H_1,\dots,H_k)$ the Ramsey-type statement that every colouring of $E(G)$ with colours $\{1 \dots k\}$ yields a monochromatic copy of $H_i$ in colour $i$ for some $i$. In the symmetric case when $H_1=\dots =H_k$, we simply write $G \rightarrow (H)_k$ and if additionally $k=2$, then we write $ G \rightarrow (H)$. Using this notation, Ramsey's theorem states that for all $k, \ell \in \N$ there exists some $n \in \N$ such that $K_n \rightarrow \left(K_{\ell}\right)_k$.

R{\"o}dl and Ruci{\'n}ski established the threshold for the property $G(n,p) \rightarrow~(H)$ which for every fixed graph $H$. 
For a graph $H=(V,E)$ we define $$d_{2}(H)=
\begin{cases}
\frac{|E|-1}{|V|-2} & \text{ if } |V|\geq 3 \land |E|>0 \\
\frac{1}{2} & \text{ if } H \cong K_2 \\
0 & \text{ if } |E|=0
\end{cases}
$$ 
and we let $m_2(H)$ denote the \textit{2-density}, defined by $m_{2}(H)= \max_{J\subseteq H} d_2(J)$.
The following is a slightly simplified version of the result mentioned 
above:

\begin{theorem}[R{\"o}dl and Ruci{\'n}ski \cite{rodl1993lower,rodl1995threshold}]\label{RRT}
Let $k\geq 2$ be an integer and let $H$ be a graph that is not a forest. Then there exist real constants $c$, $C>0$ such that
$$
\lim\limits_{n \rightarrow \infty} \Pr(G(n,p)\rightarrow(H)_{k})=
\begin{cases}
0 \text{ \, \, if } p=p(n) \leq cn^{-1/m_{2}(H)} \\
1 \text{ \, \, if } p=p(n) \geq Cn^{-1/m_{2}(H)}
\end{cases}
$$
\end{theorem}
Recently, one focus of research in this area is to establish thresholds for asymmetric Ramsey properties. Interestingly enough, some of the recent discoveries will play a key role in the proofs of our results and  will be introduced in Section~\ref{proofs}.

\subsection{Ramsey properties of randomly perturbed dense graphs}\label{subsRP}
Concerning the newer model of randomly perturbed graphs, a first reasonable question to address is whether the R{\"o}dl-Ruci{\'n}ski threshold can be improved at all.
In fact this cannot be achieved for $k\geq 3$ (i.e. more than 2 colours) and small~$\gamma $: In case of $\gamma n^2 \leq ex(n,H)$ (for example, when $\gamma<\frac{1}{4}$ and $H$ is a clique) there exists an $H$-free, $\gamma$-dense graph $G_n$. Then we can assign one colour to all edges from $G_n$ without admitting a monochromatic copy of $H$ in that colour. We still have at least two unused colours left to cope with the edges of the random graph, so we will be able to colour the remaining edges without admitting a monochromatic copy of $H$, unless we have $G(n,p) \rightarrow (H)_2$. By Theorem~\ref{RRT} a threshold for $G(n,p)\rightarrow(H)_k$ is also a threshold for $G(n,p) \rightarrow (H)_2$ and thereby for $G_n \cup G(n,p)\rightarrow (H)_k$ by the above consideration. Hence, we will focus on the case $k=2$ only. We first recall the Ramsey result from \cite{krivelevich2006smoothed}.

\begin{theorem}[Krivelevich, Sudakov and Tetali \cite{krivelevich2006smoothed}] \label{K3KtThm}
\begin{enumerate}[label=(\roman*)]
\item If $p=\omega(n^{-2/(t-1)})$, then for any $0<\gamma <1$, any integer $t \geq 3$ and any $\gamma$-dense $n$-vertex graph $G_n$ we a.a.s have  
 $$ G_n \cup G(n,p)\shortrightarrow(K_3,K_t).$$
 \item
If $p =o(n^{-2/(t-1)})$, then for any constant $0<\gamma<\frac{1}{4}$ and for every $t \geq 3$ there exists a $\gamma$-dense $n$-vertex graph $G_n$ such that we a.a.s.\ have
$$G_n \cup G(n,p)\nrightarrow(K_3,K_t).$$
\end{enumerate}
\end{theorem}
Note that in particular this shows that  $p(n)=n^{-1}$ is a threshold function for $G_n \cup G(n,p) \rightarrow (K_3)$. 
\subsection{Our results}\label{ourresults}

As mentioned above, our proofs are based on asymmetric Ramsey results and these involve the \textit{asymmetric 2-density} $m_2(G,H)$. For two graphs~$G$ and~$H$, both having at least one edge, let 
\begin{align*}
d_2(G,H)=\frac{|E(H)|}{|V(H)|-2+1/m_2(G)}
\end{align*}
and let $$m_2(G,H)=\max_{J\subseteq H} d_2(G,J)\text{\,.}$$
Then our first result reads as follows.
\begin{theorem} \label{mainresult1}
Let $\gamma>0$ be a real constant and let $r\geq 5$ be an integer.
\begin{enumerate}[label=(\roman*)]
\item \label{mainresult1even}
Let $\varepsilon >0$ be a constant and $p=\omega\left(n^{-(1-\varepsilon) /m_{2}\left(K_{\ceil{r/2}},K_{r}\right)}\right)$. For any $\gamma$-dense $n$-vertex graph $G_n$ we a.a.s.\ have $$G_{n}\cup G(n,p)\shortrightarrow(K_r).$$
\item \label{mainresult1odd}
Let $r$ be odd.
Then there exists a real constant $C>0$ such that the following holds for  $ p\geq Cn^{-1/m_{2}(K_{\frac{r+1}{2}},K_{r})}$. For any $\gamma$-dense $n$-vertex graph $G_n$ we a.a.s.\ have 
$$G_{n}\cup G(n,p)\shortrightarrow(K_r).$$
\end{enumerate}
\end{theorem} 
We complement Theorem~\ref{mainresult1} by the following lower bound for the threshold.
\begin{theorem} \label{mainresult0}
Let $r \geq 5$ be an integer and let $\gamma<\frac{1}{4}$ be a positive constant. Then there is a real constant $c>0$ and an $n$-vertex graph $G_n$ with ${|E(G_{n})| \geq \gamma n^{2}}$ such that for ${p \leq cn^{-1/m_{2}(K_{\ceil{\frac{r}{2}}},K_{r})}}$ we a.a.s.\ have
$$G_{n}\cup G(n,p)\nrightarrow(K_r). $$
\end{theorem}
Note that Theorem \ref{mainresult0} shows that Theorem \ref{mainresult1}\ref{mainresult1odd} is asymptotically optimal for odd $r\geq 5$ while for even $r\geq 6$ Theorem \ref{mainresult1}\ref{mainresult1even} leaves a `gap' of an arbitrarily small $\varepsilon$ in the exponent.
Independently of our work, Das and Treglown \cite{DasTreglown} proved a more general result which closes these gaps.
Finally, the following theorem covers the remaining case $r=4$.
\begin{theorem} \label{K4thm}
\begin{enumerate}[label=(\roman*)]
\item \label{1statementK4}
Let $\gamma >0$ and $\varepsilon>0$. Then for any graph $G_n$ on $n$ vertices with at least $\gamma n^2$ edges and any $p\geq n^{-\frac{1}{2}+\varepsilon}$ we a.a.s.\ have $G_n\cup G(n,p)\rightarrow(K_4).$ 

\item\label{0statementK4}
For $0<\gamma<\frac{1}{4}$ there exists a graph $G_n$ on $n$ vertices with at least $\gamma n^2$ edges such that for $p= o( n^{-\frac{1}{2}})$ we  a.a.s.\ have $G_n\cup G(n,p)\nrightarrow(K_4).$
\end{enumerate}
\end{theorem} 
Again, the result of Das and Treglown \cite{DasTreglown} improves the upper bound in Theorem~\ref{K4thm}\ref{1statementK4}, while Theorem~\ref{K4thm}\ref{0statementK4} improves their lower bound.
Thus, both results together establish $n^{-\frac{1}{2}}$ as a threshold for $G_n \cup G(n,p) \rightarrow (K_4)$.
\section{Preliminaries \& Notation}

In this short section we introduce further notations and a few basic results that we will use repeatedly.
For a graph $G=(V,E)$ and a subset ${U \subseteq V}$, by $G[U]$ we denote the subgraph of $G$ induced by the vertex set $U$. Furthermore, we will sometimes write $E(U)$ instead of $E(G[U])$. 

By $\rho(H)$ we denote the \textit{density} of a graph $H$ which is defined as $$\rho(H)=\frac{|E(H)|}{|V(H)|}.$$
The following well-known result establishes the threshold for containing a fixed subgraph.
\begin{theorem} [\cite{erdos1960evolution}] \label{subgraphthreshold}
Let $r\geq 3$ be an integer. 
We have 
\begin{align*}
 \lim_{n \to \infty} \Pr (G(n,p) \text{ contains a copy of $K_r$ as a subgraph})\\=
\begin{cases}
0 & \text{if } p=o(n^{-\frac{1}{\rho(K_r)}})\\
1 & \text{if } p=\omega(n^{-\frac{1}{\rho(K_r)}})
\end{cases}
\end{align*}
\end{theorem}
In order to get exponential bounds on the probability of non-existence of a graph we will use the following variant of Janson's inequality (see \cite{janson1990poisson}, \cite{janson1988exponential}).
\begin{theorem}[\cite{janson1990poisson}, \cite{janson1988exponential}]\label{Janson}
For any $r\geq 3$ there exists some constant $c_r$ such that for any $p<1$ and all $n \in \N$ the probability that $G(n,p)$ is $K_r$-free is at most $2^{-c_r n^{r}p^{r \choose 2}}.$
\end{theorem}

\section{Proofs} \label{proofs}
\subsection{Proof of Theorem \ref{mainresult0}}
We start off by showing the optimality of our result for $r\geq 5$ as this proof is easy and still already illustrates the relation between Ramsey properties of the considered model and (asymmetric) Ramsey properties of pure random graphs.
The following $0$-statement for an asymmetric Ramsey property is crucial. Note that the result of Marciniszyn, Skokan, Sp{\"o}hel and Steger \cite{marciniszyn2009asymmetric} is more general as it addresses an arbitrary number of colours and cliques.
\begin{theorem}[\cite{marciniszyn2009asymmetric}]\label{MSSST}
Let $3 \leq \ell \leq r$ be integers. Then there exists a real constant $ c>0$ such that for $p \leq cn^{-1/m_{2}(K_{\ell}, K_{r})}$ we a.a.s.\ have
${G(n,p) \rightarrow (K_{\ell},K_r)}.$
\end{theorem}
\begin{proof}[Proof of Theorem \ref{mainresult0}]
Let $G_n$ be the complete bipartite graph $K_{\ceil{\frac{n}{2}}, \floor{\frac{n}{2}}}$. Then, for sufficiently large $n \in \N$  we have
$$|E(G_n)|\geq \ceil{n/2}\floor{n/2}\geq \frac{(n+1)(n-1)}{4}\geq \gamma n^2$$ as required. We obtain $c=c(\ceil{\frac{r}{2}},r)$ from Theorem \ref{MSSST}.

We get $$\lim\limits_{n \rightarrow \infty} \Pr(G_{n}\cup G(n,p)\shortrightarrow(K_r)) \leq \lim\limits_{n \rightarrow \infty}\Pr(G(n,p) \rightarrow (K_{\ceil{\frac{r}{2}}},K_r))=0,$$
where the equality follows from Theorem \ref{MSSST} and the inequality follows from the following deterministic argument.

Suppose that $H_n \nrightarrow (K_{\ceil{\frac{r}{2}}},K_r)$ for an $n$-vertex graph $H_n$ and let $\phi$ be a red-blue-colouring of $E(H_n)$ that does not admit a red copy of $K_{\ceil{\frac{r}{2}}}$ or a blue copy of $K_r$. We extend $\phi$ to a colouring $\Phi$ of $E(G_n \cup H_n)$ by assigning the colour red to all remaining edges. The resulting colouring $\Phi$ does not admit a blue copy of $K_r$, since $\phi$ does not. Futhermore, it does not yield a red copy of $K_r$, because any such copy would need to have at least $ \ceil{\frac{r}{2}}$ vertices in one of the partition classes of the bipartite graph $G_n$, contradicting the choice of $\phi$. 
\end{proof}
\subsection{Proof of Theorem \ref{mainresult1}}

The proof of our result for cliques of even size $r\geq 6$ works as follows: 
We first show that in order to satisfy $G_n \cup G(n,p) \rightarrow (K_r)$ it suffices to have $G(n,p)[U] \rightarrow (K_{\frac{r}{2}},K_r)$ for all 'almost linear' subsets $U\subseteq V$ (those of size at least $n^{1-\varepsilon}$). In the second step we bound the probability of the above event by means of an asymmetric Ramsey result that was established recently. 

For odd $r \geq 5$ we can improve the technique by seeking for 
$G(n,p)[U] \rightarrow (K_{\frac{r+1}{2}},K_r)$ in 'linear subsets' and $G(n,p)[U] \rightarrow (K_{\frac{r-1}{2}},K_r)$ in the 'almost linear subsets'. Since the latter property generally requires a smaller $p$ to be ensured, the first will be the limiting factor, although we ask for it in slightly bigger subsets. By means of this little trick, the $\varepsilon$ reducing the size of our subsets in the second condition will no longer play a role and we get the asymptotically tight 1-statement.
A key element for the first step (of both proofs) is the following lemma given by Fox and Sudakov in their article on the method of dependent random choice \cite{fox2011dependent}:
\begin{lemma}[\cite{fox2011dependent}]\label{DRCL}
Let $a$, $d$, $m$, $n$ and $r$ be positive integers. Let $G=(V,E)$ be a graph with $|V|=n$ and average degree $d=d(G)=\frac{2 |E|}{|V|}$. If there is a positive integer t such that
\begin{align}\label{DRCI}
\frac{d^{t}}{n^{t-1}}-{n \choose r} \left(\frac{m}{n}\right)^{t}\geq a,
\end{align}
then $G$ contains a subset $U\subseteq V$ of at least $a$ vertices such that every $r$ vertices in $U$ have at least $m$ common neighbours.
\end{lemma}

\begin{corollary} \label{RDCCor}
For any $\gamma$, $\varepsilon>0$ and $r \in \N$ there is a constant $\alpha>0$ and $n_{0} \in \N$ such that the following holds for all integers $n \geq n_{0}$. For every $\gamma$-dense $n$-vertex graph $G=(V,E)$ there is a subset $U \subseteq V$ with $|U| \geq \alpha n$ such that every $r$ vertices in $U$ have at least $n^{1-\varepsilon}$ common neighbours.
\end{corollary}

\begin{proof}
The following is a straightforward application of the Lemma we just introduced. As usual, for the sake of a less baroque presentation, we do not round our parameters and instead assume they are integers.
For $t={\frac{r}{\varepsilon}}$ and $\alpha = \frac{\gamma^{t}}{2}$ let $n_{0} \in \N$ be sufficiently large such that $n_{0} \geq \frac{2}{\gamma^{t}}$.
For any integer $n\geq n_0$ and a given graph $G$ satisfying the above properties we then get
\begin{align*}
 \frac{d(G)^{t}}{n^{t-1}}-{n \choose r} \left(\frac{n^{1-\varepsilon}}{n}\right)^{t} \geq
\frac{(\gamma n)^{t}}{n^{t-1}}-{n \choose r} \left(\frac{n^{1-\varepsilon}}{n}\right)^{t} 
 \geq \gamma^{t} n - n^{r-\varepsilon t} \\ = \gamma^{t}n-1 \geq (\frac{\gamma^{t}}{2}+\frac{\gamma^{t}}{2})n-1 \geq (\alpha+\frac{1}{n_0})n-1 \geq \alpha n
\end{align*}
which verifies (\ref{DRCI}) for $m=n^{1-\varepsilon}$ and $a=\alpha n$. Hence, the corollary follows from Lemma~\ref{DRCL}.
\end{proof}
The following lemma completes step 1.

\begin{lemma}\label{mylemma}
Let $\gamma$, $\varepsilon>0$ be real constants and let $r \in \N$. Then there exist a real constant $\alpha>0$ and  $n_{0} \in \N$ such that the following holds for all integers  $n \geq n_{0}$ and all $\gamma$-dense $n$-vertex graphs $G=(V,E)$. 
\begin{equation} 
\begin{split} \label{myLemmaImpl}
 & \text{If we have }G[U] \rightarrow(K_{\ceil{\frac{r}{2}}}, K_{r}) \text{ for all subsets }U \subseteq V\text{ with }|U|\geq \alpha n \text{} \text{ and } \\
 &   G[U] \rightarrow(K_{\floor{\frac{r}{2}}}, K_{r}) \text{ for all subsets }U \subseteq V \text{ with }|U|\geq n^{1-\varepsilon},\text{ }  \text{then }G \rightarrow(K_{r}).
 \end{split}
\end{equation}
\end{lemma}

\begin{proof}
Let $\alpha=\alpha(\gamma,\varepsilon,r)$ and $n_0=n_0(\gamma,\varepsilon,r)$ be given by Corollary \ref{RDCCor}.
Let $G=(V,E)$ be a graph  satisfying $|V|=n\geq n_0$ and the other assumptions from above. Suppose for contradiction that  $\phi$ is a red-blue-colouring of $E(G)$ without a monochromatic copy of $K_{r}$. Without loss of generality we may assume $|\phi^{-1}(\{\text{red}\})|\geq \frac{\gamma}{2} n^{2}$ and let $G_{\text{red}}$ be the subgraph on $V$ that contains only the red edges. By Corollary~\ref{RDCCor} there exists $U \subseteq V$ with $|U| \geq \alpha n$ such that every $r$ ver\-ti\-ces from $U$ have at least $n^{1-\varepsilon}$ common neighbours in $G_{\text{red}}$. It follows from our assumption that we have $G[U]\rightarrow(K_{\ceil{\frac{r}{2}}}, K_{r}).$ Since $G[U]$ does not contain a blue copy of $K_{r}$, there must be a copy of $K_{\ceil{\frac{r}{2}}}$ in $G_{\text{red}}[U]$. Let $X$ be the set of vertices inducing this copy. Owing to $|X|=\ceil{\frac{r}{2}}\leq r$ and the choice of $U \supseteq X$, the common neighbourhood $W$ of $X$ in $G_{\text{red}}$ has size at least $n^{1-\varepsilon}$. Now the second assumption implies $G[W]\rightarrow(K_{\floor{\frac{r}{2}}}, K_{r})$ and by the argument given above we find a copy of $K_{\floor{\frac{r}{2}}}$ in $G_{\text{red}}[W]$. Let $Y$ be the set of vertices forming this copy; then $X \cup Y$ induces a copy of $K_{r}$ in $G_{\text{red}}$, yielding a contradiction.
\end{proof}
As indicated above, in step 2 we want to bound the probability of the event forming the hypothesis in \eqref{myLemmaImpl} using asymmetric Ramsey results.

Not too long ago, Kohayakawa, Schacht and Sp{\"o}hel \cite{kohayakawa2014upper} proved an asymmetric 1-statement for two well-behaved graphs $G$ and $H$. In particular, for cliques their upper bound asymptotically coincides with the lower bound we met in Theorem \ref{MSSST}. However, for our purposes  their main result does not help much because we need the specific exponential form of the error probability. This is because for our arguments the error probability still needs to converge to 0 when multiplied with the number of subsets of the vertex set, i.e. $2^n$.

Instead we want to apply Lemma 23 from \cite{kohayakawa2014upper} which (in the original form) contains some paper-specific notation that we do not need. Hence, we state a version that is well adapted to our setting and provide a short deduction in the following lines. However, we do not go into any details of \cite{kohayakawa2014upper}:

For our purposes we take $G=K_l$ and $H=K_r$. It is easy to check that~$K_r$ is strictly balanced with respect to $d_2(K_l,\cdot)$ (see p.3 from that paper for the definition). Then Lemma 13 from that article assures that $K_l$ and $K_r$ satisfy the hypothesis of Lemma 23.

In the lemma you find a graph parameter $x^*(H)>0$ (where, of course,~$H$ is a graph). Since we do not intend to use its explicit form anywhere, we refer to \cite{kohayakawa2014upper} (Definiton 11) in case the reader wishes to have a look at the definiton.  

\begin{lemma}[\cite{kohayakawa2014upper}]\label{KSSL}
Let $3 \leq \ell < r$ be integers. Then there exist real constants $C>0$, $ b>0$ and $n_{0} \in \N$ such that for any integer $n \geq n_{0}$ and any $p$ satisfying $$n^{-1/x^*(K_r)}\geq p \geq Cn^{-1/m_{2}(K_{\ell}, K_{r})}$$ we have
\begin{equation*}
\Pr(G(n,p)\rightarrow(K_\ell,K_r)) \geq 1-2^{-bn^{r}p^{r \choose 2}}.
\end{equation*}
\end{lemma}
Additionally, we need to know that there exists a $p$ that lies between the bounds given in Lemma \ref{KSSL}.  The following lemma ensures that and it follows from Lemma~13 (in \cite{kohayakawa2014upper}), once more using that $K_r$ is strictly balanced with respect to $d_2(K_l,\cdot)$.
\begin{lemma}[\cite{kohayakawa2014upper}]\label{x*Lemma}
For integers $3 \leq \ell < r$ we have $$m_2(K_l,K_r)<x^*(K_r).$$
\end{lemma}
With this in hand, we will be able to complete step 2.
By the union bound, the probability that the hypothesis of \eqref{myLemmaImpl} fails can now be bounded by a term of the form $$2^n 2^{-bn^{(1-\varepsilon) r} p^{r \choose 2}}$$ (or without the $\varepsilon$ in the 'improved setting') and it will turn out that this term converges to $0$ for our range of $p$. Note that this is not surprising as $n^r p^{r \choose 2}$ is of the same order as the expected number of copies of $K_r$ which should outdo~$n$ (see parts \ref{calceven} and \ref{calcodd} of Lemma \ref{Calc}). We hope that this already outlines the proofs reasonably well. Still, on the next pages we present the proofs for both cases in great detail. 
In order to make the proofs less technical, we extract some calculations revolving around the asymmetric 2-density into another lemma.
\begin{lemma}\label{Calc}
Let $r\geq5$ be an integer and let $C$, $\varepsilon>0$.
\begin{enumerate}[label=(\roman*)]
\item \label{m2clique}
For cliques the asymmetric 2-density has the form
$$m_2(K_l,K_r)=\frac{{r \choose 2}}{r-2+2/(l+1)},$$
where $3\leq l\leq r$ is another integer.
\item \label{calceven} If $p \geq C n^{-(1-\varepsilon)/m_2(K_{\ceil{r/2}},K_r)}$  and $\varepsilon< \frac{1}{6}$, then we have
\begin{align*}
n^{(1-\varepsilon)r}p^{{r \choose 2}}\geq C^{r\choose 2}n^{\frac{5}{4}}.
\end{align*}
\item \label{calcodd} If $r$ is odd, $p \geq Cn^{-1/m_2(K_{\frac{r+1}{2}},K_r)}$  and $\varepsilon<\frac{1}{4r}$, then we have 
\begin{align*}
 n^{r}p^{{r \choose 2}} \geq n^{(1-\varepsilon)r}p^{r \choose 2} \geq C^{r \choose 2}n^{\frac{5}{4}}.
\end{align*}
\item \label{calctau}
Finally, for odd $r\geq 5$ we have
$${m_{2}(K_{\frac{r+1}{2}},K_{r})}
>{m_2(K_{\frac{r-1}{2}},K_{r})}.
$$
\end{enumerate}
\end{lemma}
\begin{proof}
\begin{enumerate}[label=(\roman*)]
\item This follows from 
$$
m_2(K_l,K_r)=d_2(K_l,K_r)=\frac{{r \choose 2}}{r-2+1/m_2(K_l)}
$$
and 
\begin{align*}
m_2(K_l)= d_2(K_l)=\frac{{l \choose 2}-1}{l-2}= \frac{\frac{(l+1)(l-2)+2}{2}-1}{l-2}= \frac{l+1}{2}.
\end{align*}
\item Using \ref{m2clique} we get
\begin{align*}
 n^{(1-\varepsilon)r}p^{{r \choose 2}}&\geq C^{r \choose 2} n^{(1-\varepsilon)(r-(r-2+\frac{2}{\ceil{r/2}+1}))}\\
 &\geq C^{r \choose 2} n^{(1-\varepsilon)(2-\frac{2}{3+1})}
\geq  C^{r \choose 2} n^{\frac{5}{6}\cdot \frac{3}{2}}=C^{r \choose 2} n^{\frac{5}{4}}.
\end{align*}
\item The first inequality in the statement is trivial and the second can be verified by
\begin{align*}
 n^{(1-\varepsilon)r}p^{r \choose 2}\mathrel{\overset{\makebox[0pt]{\mbox{\normalfont\tiny\sffamily {\ref{m2clique}}}}}{\geq}} C^{r \choose 2} n^{(1-\varepsilon)r-(r-2+\frac{2}{(r+1)/2+1})}\geq C^{r \choose 2} n^{r-\frac{1}{4}-r+2-\frac{1}{2}}
 =  C^{r \choose 2} n^{\frac{5}{4}}.
\end{align*}
\item  This can be seen easily after applying \ref{m2clique} to both sides.
\end{enumerate}
\end{proof}

We start with the general case where $r\geq 5$ is an arbitrary integer and $\varepsilon>0$ is given.

\begin{proof}[Proof of Theorem \ref{mainresult1}\ref{mainresult1even}]
Throughout this proof let $m_2 \coloneqq m_2(K_{\ceil{\frac{r}{2}}},K_r)$ and \mbox{$x^* \coloneqq x^*(K_r)$}.
By monotonicity we may assume $\varepsilon \leq \frac{1}{6}$ and $(1-\varepsilon)x^*>m_2$. We start by applying our preparatory lemmas:
\begin{itemize}
\item Lemma \ref{KSSL} yields constants $C=C(\ceil{\frac{r}{2}},r)$, $b=b(\ceil{\frac{r}{2}},r)$ and $n_1=n_1(\ceil{\frac{r}{2}},r)$  such that
\begin{align} \label{Proofr}
\Pr (G(\tilde{n},\tilde{p}) \nrightarrow(K_{\ceil{r/2}}, K_{r}))\leq 2^{-b\tilde{n}^{r} \tilde{p}^{ {r \choose 2}}}
\end{align}
for all integers $\tilde{n} \geq n_1$ and any $\tilde{p}$ with $$\tilde{n}^{-1/x^*} \geq \tilde{p} \geq C\tilde{n}^{-1/m_{2}}.$$
 \item Lemma \ref{mylemma} yields $\alpha =\alpha(\gamma,\varepsilon,r)$ and $n_2 =n_2(\gamma,\varepsilon,r)$ such that \eqref{myLemmaImpl} holds for graphs $G$ with $|V(G)| \geq n_2$.
 
\end{itemize}
Let $n_0 \in \N$ satisfy $$n_0\geq \max (n_1^{1/(1-\varepsilon)},n_2, \alpha^{-1/\varepsilon})$$ and additionally 
$$n^{-1/x^*} \geq C n^{-(1-\varepsilon) /m_{2}}$$
for all integers $n \geq n_0;$
the latter is possible due to the assumption we made in the beginning of the proof.
Now let an integer $n\geq n_0$ be given.
By monotonicity we may assume $$p= C n^{-(1-\varepsilon) /m_{2}}.$$
Then for any $U \subseteq \{1,...,n\}$ with $|U| \geq n^{1-\varepsilon}$ we may apply (\ref{Proofr}) to $G(n,p)[U]$ (which we may identify with the random graph $G(|U|,p)$), since we have
\begin{equation*}
|U|\geq n^{1- \varepsilon}\geq n_0^{1- \varepsilon} \geq n_1  
\end{equation*}
and 
\begin{align*}
|U|^{-1/x^*} \geq n^{-1/x^*}  \geq p= C n^{-(1-\varepsilon) /m_{2}} \geq C|U|^{-1/m_{2}}.
\end{align*}
 \eqref{Proofr} yields 
\begin{align}\label{forfinalBlockeven}
\Pr (G(n,p)[U] \nrightarrow(K_{\ceil{r/2}}, K_{r}))\leq 2^{-bn^{(1-\varepsilon)r} p^{ {r \choose 2}}}.
\end{align}
Owing to $\alpha n = \alpha n^{\varepsilon}n^{1-\varepsilon}\geq \alpha \alpha^{(-1/\varepsilon)\varepsilon}n^{1-\varepsilon}= n^{1-\varepsilon}$,  the following statement implies \eqref{myLemmaImpl}:
\begin{align*}
& \text{If we have }G(n,p)[U] \rightarrow(K_{\ceil{r/2}}, K_{r})\text{ for all subsets }U \subseteq V \text{ with } |U|\geq n^{1-\varepsilon}, \\ & \text{then }G(n,p) \rightarrow(K_{r})\text{.}
\end{align*}
 Therefore, we have
\begin{align*}
& \Pr (G_{n}\cup G(n,p) \nrightarrow (K_{r}))  \\
\leq \,\,& \Pr( \exists \, U \subseteq V: |U|\geq n^{1-\varepsilon} \land G(n,p)[U]\nrightarrow (K_{\ceil{r/2}},K_r))\\
\leq \,\,& \sum_{U \subseteq V: |U|\geq n^{1-\varepsilon}}\Pr(G(n,p)[U]\nrightarrow (K_{\ceil{r/2}},K_r))\\
\mathrel{\overset{\makebox[0pt]{\mbox{\normalfont\tiny\sffamily {(\ref{forfinalBlockeven})}}}}{\leq}} \,\,& \sum_{U \subseteq V: |U|\geq n^{1-\varepsilon}} 2^{-b\left(n^{1-\varepsilon}\right)^{r} p^{ {r \choose 2}}}
\leq  2^{n-b C^{{r \choose 2}}n^{\frac{5}{4}}},\\
\end{align*}
where the last inequality uses Lemma \ref{Calc}\ref{calceven}.
This proves the theorem since the last term converges to $0$ when $n$ goes to infinity.
\end{proof}

Now as to the more specific case where $r\geq 5$ is odd:
\begin{proof}[Proof of Theorem \ref{mainresult1}\ref{mainresult1odd}.]
Throughout the proof let $m_2^+ \coloneqq m_2(K_{\frac{r+1}{2}},K_r)$, $m_2^- \coloneqq m_2(K_{\frac{r-1}{2}},K_r)$ and $x^* \coloneqq x^*(K_r).$
By Lemma \ref{Calc}\ref{calctau} there exists a $\tau>0$ such that $m_2^+= m_2^-\cdot\frac{1}{1- \tau}$ and we let $\varepsilon= \min(\tau,\frac{1}{4r})>0$. Again, we start by applying our preparatory lemmas: 
\begin{itemize}
\item Lemma \ref{KSSL} yields constants $C_1=C_1(\frac{r+1}{2},r)$, $b_1=b_1(\frac{r+1}{2},r)$ and $n_1=n_1(\frac{r+1}{2},r)$  such that
\begin{align} \label{Proofr+1}
\Pr (G(\tilde{n},p) \nrightarrow(K_{\frac{r+1}{2}}, K_{r}))\leq 2^{-b_1\tilde{n}^{r} p^{ {r \choose 2}}}
\end{align}
for all integers $\tilde{n} \geq n_1$ and any $p$ with $$\tilde{n}^{-1/x^*} \geq p \geq C_1\tilde{n}^{-1/m_{2}^+}.$$
\item Lemma \ref{KSSL} yields constants $C_2=C_2(\frac{r-1}{2},r)$, $b_2=b_2(\frac{r-1}{2},r)$ and $n_2=n_2(\frac{r-1}{2},r)$  such that
\begin{align} \label{Proofr-1}
\Pr (G(\tilde{n},p) \nrightarrow(K_{\frac{r-1}{2}}, K_{r}))\leq 2^{-b_2\tilde{n}^{r} p^{ {r \choose 2}}}
\end{align}
for all integers $\tilde{n} \geq n_2$ and any $p$ satisfying $$\tilde{n}^{-1/x^*} \geq p \geq C_2\tilde{n}^{-1/m_{2}^-}.$$
 \item Lemma \ref{mylemma} yields $\alpha =\alpha(\gamma,\varepsilon,r)$ and $n_3 =n_3(\gamma,\varepsilon,r)$ such that \eqref{myLemmaImpl} holds for graphs $G$ with $|V(G)|\geq n_3$. We may assume $\alpha \leq 1$.
\end{itemize}
Now let $$C=\max(C_1,C_2)\cdot\alpha^{-1/m_2^+}\geq \max(C_1,C_2) $$ and let $n_0$ satisfy $n_0\geq  \max(\frac{n_1}{\alpha},n_2^{1/(1-\varepsilon)},n_3)$ and additionally  $$n^{-1/x^*}\geq Cn^{-1/m_{2}^+}$$ for all integers $n \geq n_0$,
which is possible by Lemma \ref{x*Lemma}. Now let an integer $n \geq n_0$ be given. By monotonicity we may assume $$p= Cn^{-1/m_{2}^+}.$$ 
\begin{itemize}
\item For any $U \subseteq \{1,...,n\}$ with $|U|\geq \alpha n$ we may apply \eqref{Proofr+1} to $G(n,p)[U]$ (which we may identify with $G(|U|,p)$), as we have 
$|U|\geq \alpha n \geq n_1$ and 
\begin{align*}
 |U|^{-1/x^*} &\geq n^{-1/x^*} \geq p= Cn^{-1/m_{2}^+}\\
 & \geq C_1(\alpha n)^{-1/m_{2}^+}\geq C_1|U|^{-1/m_{2}^+}.
\end{align*}
 (\ref{Proofr+1}) shows 
\begin{align}\label{forfinalblock1}
\Pr (G(n,p)[U] \nrightarrow(K_{\frac{r+1}{2}}, K_{r}))\leq 2^{-b_1(\alpha n)^{r} p^{ {r \choose 2}}}.
\end{align}
\item 
Analogously, if $r\geq 7$, for any $U \subseteq \{1,...,n\}$ with $|U|\geq n^{1-\varepsilon}$ we may apply \eqref{Proofr-1} to $G(n,p)[U]$ since we have $|U|  \geq n^{1-\varepsilon}\geq n_2$ and 
\begin{align*}
  |U|^{-1/x^*}\geq n^{-1/x^*} \geq p  = Cn^{-1/m_{2}^+} = Cn^{-(1-\tau)/m_{2}^-} 
  &\geq C_2n^{-(1-\varepsilon)/m_{2}^-}\\ &\geq C_2|U|^{-1/m_{2}^-},
\end{align*}
where we used the definitions of $\tau$ and $\varepsilon$. We get
\begin{align}\label{forfinalblock2}
\Pr (G(n,p)[U] \nrightarrow(K_{\frac{r-1}{2}}, K_{r}))\leq 2^{-b n^{(1-\varepsilon)r} p^{ {r \choose 2}}},
\end{align}
where $b=\min(b_2, c_5)$ with $c_5$ being the constant yielded by Theorem~\ref{Janson}. This follows from (\ref{Proofr-1}) for $r\geq 7$ and from Theorem~\ref{Janson} for $r=5$. Note that in the latter case $\frac{r-1}{2}=2$, and thus $G \shortrightarrow (K_{\frac{r-1}{2}},K_r)$ is just containing a copy of $K_r$. Therefore, the asymmetric Ramsey result is not applicable in that case and we have to go for Janson's.
\end{itemize}

Using (\ref{myLemmaImpl}) we get
\begin{align*}
& \Pr (G_{n}\cup G(n,p) \nrightarrow (K_{r})) \\
 & \leq \Pr ((\exists \, U \subseteq V :|U| \geq \alpha n \land G(n,p)[U] \nrightarrow(K_{\frac{r+1}{2}}, K_{r})) \lor \\
 & \,\,\,\,\,\,\,\,\,\,\,\,\,\,\,\,\,(\exists \, U \subseteq V :|U| \geq n^{1-\varepsilon} \land G(n,p)[U] \nrightarrow(K_{\frac{r-1}{2}}, K_{r}) )) \\
  & \leq  \sum_{U \subseteq V: |U|\geq \alpha n} \Pr (G_(n,p)[U] \nrightarrow(K_{\frac{r+1}{2}}, K_{r}))+\\
  & \,\,\,\,\,\,\, \sum_{U \subseteq V: |U|\geq n^{1-\varepsilon}} \Pr (G(n,p)[U] \nrightarrow(K_{\frac{r-1}{2}}, K_{r})) \\
   & \mathrel{\overset{\makebox[0pt]{\mbox{\normalfont\tiny\sffamily {\eqref{forfinalblock1},\eqref{forfinalblock2}}}}}{\leq}}   \sum_{U \subseteq V: |U|\geq \alpha n} 2^{-b_1(\alpha n)^r p^{{ r \choose 2}}}+\sum_{U \subseteq V: |U|\geq n^{1-\varepsilon}} 2^{-b\left(n^{1-\varepsilon}\right)^{r} p^{ {r \choose 2}}}\\
  & \leq  2^{n-b_1 C^{{r \choose 2}}\alpha^{r}n^{\frac{5}{4}}}+ 2^{n-bC^{r \choose 2} n^{\frac{5}{4}}}, 
\end{align*}
where the last inequality follows from Lemma \ref{Calc}\ref{calcodd}, which is applicable due to $\varepsilon \leq \frac{1}{4r}$.
This proves the theorem because the last term goes to $0$ as~$n$ approaches infinity.
\end{proof} 

\subsection{Brief discussion of $\boldsymbol{r\leq 4}$ and the proof of Theorem \ref{K4thm}}
Considerations are of a different nature for the clique sizes ${r=3}$ and $r=4$. In these cases ${G\shortrightarrow (K_{\ceil{\frac{r}{2}}},K_r})$ comes down to containing a copy of $K_r$. Proceeding as in Theorem \ref{mainresult0}, we obtain a lower bound of $p=n^{-\frac{1}{\rho(K_r)}}$ (see Theorem \ref{subgraphthreshold}). We also obtain an upper bound based on Lemma \ref{mylemma}; however, we need~$p$ to be of order $$n^{-\frac{1}{\rho(K_r)}+\frac{1}{|E(K_r)|}}$$ as we can only apply this technique if the expected number of copies of $K_r$ is linear.
This leaves us with a significant gap between the bounds. For $r=3$ the threshold coincides with the lower bound (see Theorem \ref{K3KtThm}), whereas for $r=4$ we will improve the lower bound so that it matches the order of the upper bound.
Let us start with a brief proof of the upper bound; the method is the same as in the proof of Theorem~\ref{mainresult1} where we went into great detail.
\begin{proof}[Proof of Theorem \ref{K4thm}\ref{1statementK4}]
Following the lines of the proof of Theorem~\ref{mainresult1}\ref{mainresult1even} and applying Theorem~\ref{Janson} yields
\begin{align*}
 \Pr(G_n \cup G(n,p)\nrightarrow (K_4))
&\leq \Pr(\exists U \subseteq V : |U| \geq n^{1- \varepsilon} \land G(n,p)[U] \text{ is } K_4 \text{-free})\\
 &\leq 2^n 2^{-c_r n^{(1-\varepsilon)4}n^{(-\frac{1}{2}+\varepsilon)6}}= 2^{n-c n^{1+2 \varepsilon}},
 \end{align*}
 for a sufficiently large $n$ and an appropriate constant $c>0$. Clearly, the last term converges to $0$ as desired.
\end{proof}
We now turn our attention to the lower bound and start by introducing our key lemma which we will prove later.  We write $G\rightarrow (H)^v$ if every 2-colouring of $V(G)$ admits a monochromatic copy of $H$.
\begin{lemma}\label{basicLemmaK_4}
Let $G=(V_1 \mathop{\dot{\cup}} V_2,E)$ be a graph such that $G[V_1]\nrightarrow (K_3)$, ${G[V_2]\nrightarrow (K_4)^v}$ and $G[V_2]$ does not contain any subgraph $H$ with $\rho(H)\geq \frac{1}{2}$ and $|V(H)|\leq 8$. Then we have $G\nrightarrow (K_4,K_4)$.
\end{lemma}
Let us now first show that Lemma~\ref{basicLemmaK_4} implies Theorem~\ref{K4thm}\ref{0statementK4}. We need the following special case of a result on the threshold for vertex-Ramsey properties by {\L}uczak, Ruci{\'n}ski and Voigt \cite{luczak1992ramsey}.
\begin{theorem}\label{vertexRamseyK4}
There exists a real constant $c>0$ such that for $p\leq cn^{-\frac{1}{2}}$ we have
$$
\lim\limits_{n \rightarrow \infty} \Pr(G(n,p)\shortrightarrow(K_4)^v)=0.
$$
\end{theorem}
\begin{proof}[Proof of Theorem \ref{K4thm}\ref{0statementK4}] 
As in the proof of Theorem~\ref{mainresult0} we choose  $G_n=K_{\ceil{\frac{n}{2}},\floor{\frac{n}{2}}}$ and we denote the partition classes by~$V_1$ and~$V_2$.
We show that $H_n=G_n \cup G(n,p)$ a.a.s.\ satisfies the hypothesis of Lemma~\ref{basicLemmaK_4}. 
By Theorem~\ref{RRT} we a.a.s.\ have ${G(n,p)\nrightarrow (K_3)}$ and thus $H_n[V_1] \nrightarrow (K_3).$ Theorem~\ref{vertexRamseyK4} ensures ${G(n,p)\nrightarrow (K_4)^v}$ and thus $H_n[V_2]\nrightarrow (K_4)^v$. Finally, by Theorem \ref{subgraphthreshold} $G(n,p)$ (and thus $H_n[V_2]$) a.a.s.\ does not contain any of the finitely many graphs $H$ with $\rho(H)\geq 2$ and $|V(H)|\leq 8$ as a subgraph. Hence the theorem follows from Lemma~\ref{basicLemmaK_4}.
\end{proof}

It remains to prove Lemma~\ref{basicLemmaK_4}. We let $${\mathcal{C}_G=\lbrace M \subseteq V^{(4)}\, |\, G[M] \text{ is a copy of }K_{4}\rbrace }$$ be the family of vertex subsets inducing a copy of $K_4$ in a graph $G$. 
\begin{proof}[Proof of Lemma~\ref{basicLemmaK_4}]
We will define a red-blue-colouring $\phi$ of $E(G)$ in three steps.
By assumption, there is a colouring $\phi_1$ of $E(V_1)$ which yields no monochromatic copies of $K_3$.

There also exists a vertex colouring of $V_2$ that does not admit monochromatic copies of $K_4$. Let $U \subseteq V_1$ be one of its two colour classes, then $G[U]$ is $K_4$-free and for any $L \in \mathcal{C}_{G[V_2]}$ we have $U \cap L\neq \varnothing$. For any such $L$ we choose an arbitrary vertex $a_L \in U \cap L$ and an arbitrary edge $e_L \in E(L)$ incident with $a_L$. Let $$E_s=\bigcup\limits_{L \,\in \, \mathcal{C}_{G[V_2]}} e_L.$$

Now we define a colouring $\phi_2$ of $E(V_2)$ as follows. For $e=vw \in E(V_2)$ we let
\begin{align*}
\phi_2(e)=
\begin{cases}
\text{red, if } v,w \in U \lor e \in E_s \\
\text{blue, otherwise}
\end{cases}
\end{align*}
Finally, for an edge $e=av$ with $a \in V_1$, $v \in V_2$ define
\begin{align*}
\phi_3(e)=
\begin{cases}
\text{blue, if } v \in U\\
\text{red, otherwise}
\end{cases}
\end{align*}
We claim that $\phi=\phi_1 \cup \phi_2 \cup \phi_3$ does not yield any monochromatic copy of $K_4$.

Suppose first that $\phi$ admits a blue copy $L$ of $K_4$.
Then, owing to our choice of $\phi_1$, we have $|L\cap V_1| \leq 2$. On the other hand, $|L\cap V_2| \leq 3$, since $L \subseteq V_2$ would imply that $e_L \in E(L)$ is a red edge. It follows from the two above observations that there are two distinct vertices $v,w \in L \cap V_2$ and $a \in L \cap V_1$. Since $av$ and $aw$ are blue, it follows from the defintion of $\phi_3$ that $v,w \in U$. However, this means that $vw$ is red in $\phi_2$, yielding a contradiction.

Let us now suppose that  $L \in \mathcal{C}_G$ is a red copy.
Again, we have $|L\cap V_2| \geq 2$ due to our choice of $\phi_1$. In case of $|L\cap V_2| \leq 3$ we again find two distinct vertices $v,w \in L\cap V_2$ and $a \in L \cap V_1$. Since $vw$ is red, by the defintion of $\phi_2$, at least one of the vertices~$v$ and~$w$ is in~$U$, and thus either~$va$ or~$wa$ is blue.

It remains to check the case $|L\cap V_1| =4$. Recall that there are no copies of $K_4$ in~$G[U]$. Therefore, there exists a vertex $z \in L\setminus U$.  Hence, all edges in $E(L)$ which are incident with $z$ were chosen to be $e_{L'}$ for some $L' \in \mathcal{C}_{G[V_2]}$. This yields two further sets $L_1, L_2 \in \mathcal{C}_{G[V_2]}$ with 
$|L \cap L_i| \geq 2$ for $i=1,2$ and $L \neq L_1 \neq L_2 \neq L$.

We will show that for $M= L \cup L_1 \cup L_2$ we have
\begin{align}\label{EM2M}
|E(M)|\geq 2|M|,
\end{align}
contradicting the last assumption of the lemma because of $|M| \leq 8$. Let 
\begin{align*}
H=G[L], \, H_1=G[L] \cup G[L_1], \, H_2=H_1 \cup G[L_2], \, s_1=|L_1 \setminus L|, \, 
s_2=|L_2 \setminus (L \cup L_1)|
\end{align*}
and similarly $$t_1=|E(H_1) \setminus E(H)|, \, t_2=|E(H_2) \setminus E(H_1)|.$$
Let us prove that 
\begin{align}\label{tisi}
 t_i\geq 2 s_i+1 \text{ for } i \in \{1,2\}.
\end{align}
This inequality can be obtained immediately for any $i$ with $s_i \in \{1,2\}$ by counting the edges incident to the vertices counted by $s_i$.
We have $0 \leq s_1,s_2\leq 2$ and $s_1 \neq 0$, hence for the proof of \eqref{tisi} it is left to check $i=2$ in case of $s_2=0$. It  is easy to confirm that $\mathcal{C}_{H_1}=\{L,L_1\}$, thus $L_2$ induces an edge $e \notin E(H_1)$, i.e. $t_2\geq 1=2s_2+1$. This completes the proof of \eqref{tisi} which yields \eqref{EM2M} due to
$$|E(M)|=6+t_1+t_2\geq 6+ 2 s_1+1 + 2s_2 +1=2(4+s_1+s_2)=2|M|.$$ 
\end{proof}

\section*{Acknowledgements}
I am very grateful to my supervisor Mathias Schacht for his time and encouragement.
I would like to thank Shagnik Das and Andrew Treglown for sharing a simplification that strengthened part \ref{0statementK4} of Theorem \ref{K4thm}.

\bibliography{bibthesis}
\bibliographystyle{SIAM}

\end{document}